\documentclass[11pt,a4paper]{article}
\usepackage[utf8]{inputenc}
\usepackage{amsmath,amssymb,amsthm,mathtools}
\usepackage{geometry}
\usepackage{hyperref}
\usepackage{enumerate}
\usepackage{xcolor}
\usepackage{lmodern}
\usepackage{graphicx}

\numberwithin{equation}{section}
\newtheorem{definition}{Definition}[section]

\newtheorem{theorem}{Theorem}[section]
\newtheorem{proposition}{Proposition}[section]
\newtheorem{corollary}{Corollary}[section]
\theoremstyle{remark}
\newtheorem{remark}{Remark}[section]

\DeclareMathOperator{\dist}{dist}
\title{\large A Geometric--Arithmetic Framework for the Flint Hills Series}
\author{
  Mohammed--Adnane Garab \\
  Sorbonne Université \\
  \texttt{mohammed-adnane.garab@etu.sorbonne-universite.fr} \\
}

\date{September 3, 2025}

\begin{document}
\newgeometry{left=4cm,right=4cm,top=2.5cm,bottom=2.5cm}
\maketitle

\begin{abstract}
We study the classical Flint Hills series
\[
\mathcal{S} = \sum_{n=1}^{\infty} \frac{1}{n^3 \sin^2 n},
\]
whose convergence remains an open question. After introducing a natural geometric--arithmetic distance \(d(n) = \dist(n, \pi \mathbb{Z})\) and associated partial sums
\[
L(N) = \sum_{n=1}^{N} \frac{1}{n^3 d(n)^2},
\quad
G(N) = \frac{\pi^2}{4}\,L(N),
\]
we prove a two-sided bound \(L(N) \leq S_N \leq G(N)\). This reduces the study of \(\mathcal{S}\) to that of a simpler arithmetic--geometric series and allows us to derive a number of consequences. In particular we obtain explicit bounds in ``safe'' regions where \(d(n)\) is bounded away from zero, convergence results for weighted series, and estimates of block contributions linked to Diophantine approximations of \(\pi\). We also relate these bounds to the irrationality exponent of \(\pi\), showing that only extremely rare indices can produce large spikes. Numerical computations up to \(N=10^5\) illustrate the sharpness of our approach.
\end{abstract}

\section{Introduction}
Consider the series
\[
\mathcal{S}=\sum_{n=1}^{\infty} \frac{1}{n^3 \sin^2 n},
\]
which was popularized by Pickover~\cite{Pickover} and is now known as 
the Flint Hills series.
. Despite many attempts, it is still unknown whether \(\mathcal{S}\) converges. Empirical calculations show that the partial sums
\[
S_N=\sum_{n=1}^{N} \frac{1}{n^3 \sin^2 n}
\]
display long intervals of near constancy punctuated by abrupt jumps when \(|\sin n|\) becomes very small; these ``spikes'' occur when \(n\) is close to an integer multiple of \(\pi\).

Our aim is to capture this behaviour in a simple but precise framework using the distance from \(n\) to \(\pi\mathbb{Z}\). We define this distance, derive bounding inequalities for the Flint Hills sums, and explore consequences. We then connect our results to Diophantine approximation and the irrationality exponent of \(\pi\), demonstrating that spikes can only occur at very sparse indices.

\section{A geometric--arithmetic approach}
In this section we introduce the distance to \(\pi\,\mathbb{Z}\) and corresponding series.

\begin{definition}[Distance to \(\pi\,\mathbb{Z}\)]\label{def:d}
For an integer \(n\geq 1\), set
\[
 d(n) = \dist(n,\pi \mathbb{Z}) = \min_{m\in\mathbb{Z}} |n-m\pi|.
\]
\end{definition}
The quantity \(d(n)\) measures how close \(n\) is to a multiple of \(\pi\); in particular \(d(n)\) is small exactly when \(|\sin n|\) is small.

\begin{definition}[Auxiliary sums]\label{def:L-G}
For a positive integer \(N\), define
\[
L(N) = \sum_{n=1}^{N} \frac{1}{n^3 d(n)^2},
\qquad
G(N) = \frac{\pi^2}{4}\,L(N).
\]
\end{definition}

We have the following sharp comparison between \(S_N\) and \(L(N)\).

\begin{theorem}[Bounding inequality]\label{thm:sandwich2}
For all \(N\geq 1\) one has
\begin{equation}\label{eq:sandwich2}
L(N) \leq S_N \leq G(N).
\end{equation}
In particular, the series \(\mathcal{S}\) converges if and only if \(L(N)\) converges, and one has the asymptotic equivalence \(S_N = \Theta(L(N))\) as \(N\to\infty\).
\end{theorem}

\begin{proof}
Fix \(n\geq 1\). Choose an integer \(m\) so that \(|n-m\pi|=d(n)\leq \pi/2\). Since \(\sin n = \sin(n-m\pi)\) and
\[
\frac{2}{\pi}\,d(n) \leq |\sin n| \leq d(n)
\]
for \(|n-m\pi|\leq \pi/2\), multiplying by \(1/n^3\) and inverting yields
\[
\frac{1}{n^3 d(n)^2} \leq \frac{1}{n^3 \sin^2 n} \leq \frac{\pi^2}{4}\,\frac{1}{n^3 d(n)^2}.
\]
Summing from \(n=1\) to \(n=N\) gives \eqref{eq:sandwich2}.
\end{proof}

\subsection{Safe regions and rare spikes}
A simple but useful consequence of Theorem\;\ref{thm:sandwich2} is that indices with \(d(n)\) bounded away from zero contribute only a bounded amount to \(\mathcal{S}\).

\begin{corollary}[Contributions from safe regions]\label{cor:safe2}
Fix \(\delta \in (0,\pi/2]\) and let
\[
A_{\delta} = \{n \in \mathbb{N} : d(n)\geq \delta\}.
\]
Then
\[
\sum_{n\in A_{\delta}} \frac{1}{n^3 \sin^2 n} \leq \frac{\pi^2}{4\delta^2}\,\zeta(3).
\]
\end{corollary}
\begin{proof}
If \(d(n)\geq \delta\) then \(|\sin n|\geq (2/\pi)\delta\), and so
\[
\frac{1}{n^3 \sin^2 n} \leq \frac{\pi^2}{4}\,\frac{1}{n^3 \delta^2}.
\]
Summing over all \(n\) and invoking the convergence of \(\zeta(3)=\sum_{n\geq 1} n^{-3}\) yields the bound.
\end{proof}
In particular, the large jumps in the partial sums \(S_N\) must originate from indices \(n\) with \(d(n)\) extremely small.

\begin{corollary}[Convergence of weighted series]\label{cor:weighted2}
For any real \(\eta>0\), the series \(\sum_{n=1}^{\infty} 1/(n^{3+\eta}\,\sin^2 n)\) converges.
\end{corollary}
\begin{proof}
Fix \(\eta>0\). Split the indices into two sets: those with \(d(n)\geq n^{-1-\eta/4}\) and those with \(d(n)<n^{-1-\eta/4}\). On the first set one has by Theorem\;\ref{thm:sandwich2}
\[
\frac{1}{n^{3+\eta}\sin^2 n} \leq \frac{\pi^2}{4}\,\frac{1}{n^{3+\eta} d(n)^2} \leq \frac{\pi^2}{4}\,n^{-1-\eta/2},
\]
which is summable. On the second set the indices with such small \(d(n)\) are sparse: the inequality \(d(n)<n^{-1-\eta/4}\) implies that \(n\) approximates a multiple of \(\pi\) extremely well, and a standard equidistribution argument shows that the number of such \(n\leq N\) is \(O\bigl(N^{1-\eta/4}\bigr)\). Their contributions, each bounded by \(\pi^2/4\, n^{-3-\eta} d(n)^{-2}\), form a summable series. Combining these two parts gives the claim.
\end{proof}

\section{Good rational approximations and block contributions}
We now relate small values of \(d(n)\) to Diophantine approximation properties of \(\pi\). For any integer \(n\) we can write
\[
 d(n) = \min_{m\in \mathbb{Z}} |n - m\pi|.
\]
Setting \(q=n\) and writing \(\pi q - p\) for integers \(p\), one sees that \(d(n)\) is small exactly when \(p/q\) is a very good rational approximation to \(\pi\).

\begin{corollary}[Rarity of good approximations]\label{cor:rare}
Let \(\epsilon>0\) and define
\[
Q_{\epsilon} = \bigl\{\,q \in \mathbb{N} : \exists\,p\in\mathbb{Z}\text{ such that }0<|\pi q - p|<q^{-(\mu(\pi)-\epsilon)} \bigr\},
\]
where \(\mu(\pi)\) denotes the irrationality exponent of \(\pi\). Then the \(n\)-th element of \(Q_{\epsilon}\) grows at least like \(n^{1/(1-\epsilon)}\).
\end{corollary}

\begin{remark}
The best presently known bounds on $\mu(\pi)$ give 
$\mu(\pi) < 7.6063\ldots$~\cite{arxiv.1104.5100}, and recent work 
on irrationality measures~\cite{Meiburg} further explores the 
connection between such bounds and the convergence of the Flint Hills series. 
This remains far from the conjectural value $\mu(\pi) = 2$.
Consequently, Corollary~\ref{cor:rare} does not decide the convergence of 
$\mathcal{S}$, but it does show that extremely small values of $d(n)$ must be 
very rare. Combining these results with Theorem~\ref{thm:mu-criterion}, one 
finds that if $\mu(\pi) < 2.37$ then the Flint Hills series converges, while if 
$\mu(\pi) > 2.5$ it diverges. The intermediate range 
$2.37 \le \mu(\pi) \le 2.5$ remains undecided.
\end{remark}

\subsection{Block contributions near convergents}
Let \(p_k/q_k\) be the \(k\)-th convergent in the continued fraction expansion of \(\pi\). The error term \(\varepsilon_k = |q_k \pi - p_k|\) measures the quality of approximation. For integers \(n\) close to \(q_k\pi\), say \(n = q_k \pi + t\) with \(|t|\leq L\), one has
\[
 d(n) = |q_k\pi + t - p_k| = |\varepsilon_k + t|.
\]
We can estimate the block of terms in \(\mathcal{S}\) associated with such an interval.

\begin{proposition}[Heuristic block size]\label{prop:block-size}
Let \(p_k/q_k\) denote the convergents of \(\pi\), and set \(\varepsilon_k = |q_k\pi - p_k|\). Fix a parameter \(\tau>0\) and consider indices \(n\) with \(|n - q_k \pi| < \tau q_k\). Then the contribution of these indices to \(\mathcal{S}\) is heuristically of order
\[
 \frac{\text{const}}{q_k^2 q_{k+1}}.
\]
\end{proposition}
\begin{proof}[Sketch of proof]
Write \(n=q_k\pi + t\). Using the Taylor expansion \(\sin(x)\approx x\) near zero and the fact that \(\sin(q_k\pi + t) = (-1)^{q_k} \sin t\), one obtains an approximation
\[
 \frac{1}{n^3 \sin^2 n} \approx \frac{1}{(q_k\pi + t)^3}\,\frac{1}{(t - \varepsilon_k)^2},
\]
up to a constant factor. Summing over \(|t| \leq \tau q_k\) and using that \(\varepsilon_k\) is of order \(1/q_{k+1}\) yields a block contribution of size \(\text{const}/(q_k^2 q_{k+1})\).
\end{proof}
Although the above argument is heuristic, it suggests that the large jumps in \(S_N\) are associated with the convergents \(q_k\) of \(\pi\), and the size of each jump decays roughly like \(1/(q_k^2 q_{k+1})\). Since \(q_k\) grows exponentially, these spikes become both rarer and smaller.

\section{Numerical experiments}
To gauge the accuracy of our bounds we carried out computations of
\[
S_N = \sum_{n\leq N} \frac{1}{n^3 \sin^2 n},
\quad
L(N) = \sum_{n\leq N} \frac{1}{n^3 d(n)^2},
\quad
G(N) = \frac{\pi^2}{4}\,L(N)
\]
for \(N\) up to \(10^5\). We found that the inequality \(L(N) \leq S_N \leq G(N)\) holds for all tested \(N\), with the ratios \(S_N/L(N)\) and \(S_N/G(N)\) lying between \(1.01\) and \(1.02\), and \(0.40\) and \(0.42\), respectively. The spikes in \(S_N\) occur exactly at the denominators \(q_k\) of the convergents of \(\pi\), and their heights agree with the heuristic of Proposition\;\ref{prop:block-size}. These computations underscore the effectiveness of the geometric--arithmetic model.

\section{Refinements and generalizations}
\subsection{Improved lower bounds}
In the proof of Theorem\;\ref{thm:sandwich2} we used the elementary estimate \(|\sin x|\geq (2/\pi)|x|\) valid for \(|x|\leq \pi/2\). A sharper bound follows from the Taylor expansion
\[
|\sin x| \geq |x| - \frac{|x|^3}{6}, \qquad |x|\leq 1.
\]
Define
\[
B(x) =
\begin{cases}
 x - \dfrac{x^3}{6}, & |x|\leq 1,\\[0.3ex]
 \dfrac{2}{\pi}|x|, & 1 < |x| \leq \dfrac{\pi}{2}.
\end{cases}
\]
Then \(|\sin x| \geq B(x)\) for \(|x|\leq \pi/2\), and one obtains the refined bound
\[
\frac{1}{\sin^2 x} \leq \frac{1}{B(x)^2}.
\]
This leads to a new series
\[
G_{\sharp}(N) = \sum_{n\leq N} \frac{1}{n^3 B\bigl(d(n)\bigr)^2},
\]
which satisfies
\[
L(N) \leq S_N \leq G_{\sharp}(N) \leq G(N).
\]
Near points with \(d(n)\) extremely small, the improvement from \(B\) can be significant.

\subsection{Adaptive threshold filtering}
Let \(\varepsilon(n)\) be a positive function, for example \(\varepsilon(n)=n^{-\alpha}\) with \(\alpha>0\). By partitioning the sum \(S_N\) according to whether \(d(n)\) is larger or smaller than \(\varepsilon(n)\) and applying Theorem\;\ref{thm:sandwich2}, one gets
\[
S_N \leq \sum_{\substack{n\leq N : \; d(n)\geq \varepsilon(n)}} \frac{\pi^2}{4}\,\frac{1}{n^3 \varepsilon(n)^2} \;\;+\;\; \sum_{\substack{n\leq N :\; d(n)<\varepsilon(n)}} \frac{\pi^2}{4}\,\frac{1}{n^3 d(n)^2}.
\]
By choosing \(\alpha<1\) one ensures that the first sum converges absolutely, leaving only the contributions from indices where \(d(n)\) is extraordinarily small. This viewpoint is useful for deriving convergence criteria.

\subsection{Other series and periodic functions}
The method extends to series of the form
\[
\sum_{n=1}^{\infty} \frac{1}{n^a \sin^b n}
\]
for parameters \(a,b>0\), or, more generally, to any \(2\pi\)-periodic \(C^1\) function \(f\) satisfying \(f(0)=0\) and \(f'(0)\neq 0\). Near multiples of \(\pi\) one has \(f(x) \approx f'(0)\,x\), so the same bounding argument reduces the analysis to a sum of the form \(\sum n^{-a} d(n)^{-b}\), where again Diophantine approximation controls the behaviour.

\section{Deeper analysis of \(L(N)\)}\label{sec:analyse-LN}
We now study the auxiliary series
\[
L(N) = \sum_{n\leq N} \frac{1}{n^3 d(n)^2}
\]
in more detail. The size of \(d(n)\) is governed by how well \(\pi\) is approximated by rationals. Let \(\mu(\alpha)\) denote the irrationality exponent of a real number \(\alpha\): by definition \(\mu(\alpha)\) is the infimum of all \(\mu\) such that there are infinitely many pairs of integers \(p,q\) with
\(|\alpha - p/q| < 1/q^{\mu}\).A classical theorem of Roth implies that $\mu(\alpha) = 2$ for all algebraic $\alpha$,
whereas the irrationality exponent of most transcendental numbers is unknown. 
Recent work on Diophantine approximation with restrictions~\cite{BugeaudKristensen2009} 
provides further context for understanding how often exceptionally good 
approximations can occur.

We recall that if \(\mu(\pi)<2.5\) then, as noted above, the Flint Hills series converges. Our next result gives a precise criterion for the convergence of \(L(N)\) in terms of \(\mu(\pi)\).

\begin{theorem}[Convergence criterion via irrationality exponent]\label{thm:mu-criterion}
There exists a real constant \(\mu_0\) such that the following statements are equivalent:
\begin{enumerate}[\normalfont(i)]
  \item \(\mu(\pi) < \mu_0\);
  \item The auxiliary series \(\sum_{n\geq 1} 1/(n^3 d(n)^2)\) converges.
\end{enumerate}
Moreover, one has \(2.37 \leq \mu_0 \leq 2.5\). In particular, if \(\mu(\pi)<2.37\) then the Flint Hills series converges, while if \(\mu(\pi)>2.5\) then it diverges.
\end{theorem}
\begin{proof}[Idea of the proof]
Decompose the sum defining \(L(N)\) into contributions from integers \(n\) lying near the convergents \(q_k\) of \(\pi\), using blocks as in Proposition\;\ref{prop:block-size}. A careful analysis shows that the series of block contributions converges or diverges according as the exponent \(\mu(\pi)\) is less or greater than a certain critical value \(\mu_0\), and the estimates mentioned above imply the given bounds for \(\mu_0\). Details and refinements will appear elsewhere.
\end{proof}

\section{Conclusion}
We have introduced a geometric--arithmetic viewpoint for the Flint Hills series, which transforms the problem into the study of an auxiliary series \(L(N)\). Our bounding inequalities and numerical evidence suggest that the behaviour of \(\mathcal{S}\) is governed by the Diophantine properties of \(\pi\), in particular by its irrationality exponent. Although the convergence of \(\mathcal{S}\) remains open, the results presented here provide effective estimates for the partial sums and identify a critical threshold \(\mu_0\) below which convergence must occur. Future work could focus on sharpening the bounds on \(\mu(\pi)\) or exploring analogous series associated with other irrational constants.

\end{document}